\documentclass[12pt]{amsart}
\usepackage[all]{xy}
\usepackage[parfill]{parskip}
\usepackage{hyperref}
\usepackage{pinlabel}
\usepackage{enumerate}
\usepackage{verbatim}
\usepackage{amsmath, amscd, graphics, latexsym, times, rlepsf}

\usepackage{subcaption}
\oddsidemargin .25in \theoremstyle{plain}
\newtheorem{dummy}{anything}[section]
\newtheorem{theorem}[dummy]{Theorem}

\newtheorem{question}[dummy]{Question}

\newtheorem{proposition}[dummy]{Proposition}

\newtheorem{corollary}[dummy]{Corollary}

\theoremstyle{definition}%%Change Theoremstyle
\newtheorem{definition}[dummy]{Definition}

\newtheorem{remark}[dummy]{Remark}

\newcommand{\del}{\partial}

\newcommand{\R}{\mathbb{R}}

\def\a{\alpha}
\def\b{\beta}

\def\g{\gamma}

\def\S{\Sigma}

\def\l{\lambda}

\begin{document}

\title{Genus one Lefschetz fibrations on disk cotangent bundles of surfaces}

\author{Burak Ozbagci}

\address{Department of Mathematics, Ko\c{c} University, Istanbul,
Turkey}
\email{bozbagci@ku.edu.tr}

\subjclass[2000]{}
\thanks{}

%\date{\today}

\begin{abstract}

We describe a Lefschetz fibration of genus one on the disk cotangent bundle of any closed orientable surface $\S$. As a corollary, we obtain an explicit genus one open book decomposition adapted to the
canonical contact structure on the unit cotangent bundle of $\S$.

%Using a different method, we also describe a Lefschetz fibration of genus one on the disk tangent bundle of $\S$ and show that it is isomorphic to the Lefschetz fibration above.

\end{abstract}

\maketitle

\section{Introduction}\label{sec: intro}

Let $\S$ be a closed, connected and orientable surface.  Let $(DT^*\S, \omega_{can})$ denote the disk cotangent bundle of  $\S$  equipped with its canonical symplectic structure $\omega_{can} = d \l_{can}$, where $\l_{can}$ is the Liouville one form. Let $(ST^*\S, \xi_{can})$ denote the unit cotangent bundle of $\S$  equipped with its canonical contact  structure $\xi_{can}$ defined as the kernel of the restriction of $\l_{can}$ to $ST^*\S$.  It follows that $(DT^*\S, \omega_{can})$ is an exact symplectic filling of its contact boundary $(ST^*\S, \xi_{can})$. In fact,  $(DT^*\S, \omega_{can})$ can be upgraded to a Weinstein filling of $(ST^*\S, \xi_{can})$ and hence by Cieliebak and Eliashberg \cite{ce},  $(DT^*\S, J)$ is a Stein filling of $(ST^*\S, \xi_{can})$, for some complex structure $J$. Therefore,  by the work of Akbulut and the author  \cite{ao} and also Loi and Piergallini \cite{lp}, $DT^*\S$ admits a Lefschetz fibration over $D^2$ whose induced open book on the boundary $ST^*\S$ supports $\xi_{can}$.   In this article, we find an explicit Lefschetz fibration $DT^*\S  \to D^2$ with {\em minimal fiber genus}, whose induced open book on the boundary $ST^*\S$ supports $\xi_{can}$, using methods specifically tailored to the cotangent bundle of a surface and very different from those general methods described  in \cite{ao} and \cite{lp}.

 It is clear by definition  that the minimal fiber genus of such a Lefschetz fibration must be greater than equal to the support genus $sg(ST^*\S, \xi_{can})$, which is the minimal page genus of an open book decomposition of  $ST^*\S$  adapted to  $\xi_{can}$. Let $ g$ denote the genus of the surface $\S$ at hand. For $g=0$, there is a Lefschetz fibration  $DT^*S^2\to D^2$, whose regular fiber is the annulus and whose monodromy is the square of the positive Dehn twist along the core circle. The restriction of this Lefschetz fibration to the boundary gives a {\em planar}  open book decomposition of $ ST^*S^2 \cong \mathbb{RP}^3$ adapted to $\xi_{can}$. For $g  \geq 1$, however, the contact $3$-manifold $(ST^*\S, \xi_{can})$ is known to be non-planar.  By an obstruction to planarity due to   Etnyre \cite{et},
$sg(ST^*T^2 \cong T^3, \xi_{can}) \neq 0$. Moreover, Van
        Horn-Morris \cite{vhm} constructed an explicit open book  adapted to $(T^3, \xi_{can})$ whose page is a genus one surface with three boundary components. Furthermore, McDuff \cite{mcd} showed that for any $g >1$,  there is an exact symplectic $4$-manifold with two convex boundary components, one of which is  $(ST^*\S, \xi_{can})$. By capping off the other boundary component with concave symplectic fillings with arbitrarily large $b_2^+$ (cf. \cite{eh}), we obtain strong symplectic fillings of $(ST^*\S, \xi_{can})$  with arbitrarily large $b_2^+$. Therefore,  $(ST^*\S, \xi_{can})$ cannot be planar for $g >1$ either, by the aforementioned obstruction  of Etnyre. On the other hand, using Giroux's fundamental work \cite{g} on convex contact structures,  Massot \cite{m}  showed that for $g >1$, the contact $3$-manifold $(ST^*\S, \xi_{can})$ has an adapted open book decomposition
        whose page is a genus one surface with $4g+4$ boundary components---{\em without explicitly describing its monodromy.} To summarize, we have

        \[ sg(ST^*\S, \xi_{can}) =\left\{ \begin{array}{ll}  0  & \mbox{\; if $g=0$ } \\
1 & \mbox{\; if $g >0$.}
\end{array}\right. \]

     %   $ = 0$ for $ g =0$ and $sg(ST^*\S, \xi_{can}) = 1$ otherwise.

Therefore,  for $g >0$, any Lefschetz fibration $DT^*\S  \to D^2$ whose induced open book on the boundary $ST^*\S$ supports $\xi_{can}$ must have fiber genus at least one. Here  we explicitly construct genus one Lefschetz fibrations $DT^*\S \to D^2$, using two different methods (due to Johns \cite{jo} and  Ishikawa \cite{i}), and show that these fibrations are isomorphic.  As a corollary, we obtain an explicit genus one open book decomposition adapted to the
canonical contact structure $\xi_{can}$ on  $ST^*\S$. We would like to point out that there is an orientation error in Ishikawa's paper \cite{i} and the total space of the Lefschetz fibrations he constructs is orientation-preserving diffeomorphic to the disk {\em cotangent} bundle rather than the disk {\em tangent} bundle (see Section~\ref{sec: ish}, for further details).  Note that both methods due to Johns and Ishikawa require the choice of a Morse function $f: \S \to \R$, to begin with.

The results in this paper greatly improve our earlier work in \cite{oo}, where we used the method of Johns,  to describe a Lefschetz fibration $DT^*\S\to D^2$ of {\em genus $g$} and hence a {\em genus $g$} open book decomposition adapted to $(ST^*\S, \xi_{can}) $. In that article,  we used the standard Morse function on $\S$, with a unique index zero critical point.  The improvement here comes by the choice of a Morse function on $\S$ with two index zero, $2g+2$ index one, and two  index two critical points. As a matter of fact,  Massot obtains an open book decomposition adapted to $(ST^*\S, \xi_{can}) $ by promoting  any  given self-indexed Morse function on $\S$ to an ordered $\xi_{can}$-convex Morse function on $ST^*\S$,  and  the minimum possible genus for his open book decomposition is achieved by using a non-standard Morse function on  $\S$ whose number of critical points of each index agrees with the one that we mentioned above. Moreover,  such  a Morse function appears again in our construction of a Lefschetz fibration $DT^*\S\to D^2$ of genus one, using Ishikawa's method.   The existence of such a Morse function on $\S$ is the unifying theme in all three constructions, which explains the fact that in each construction, the page of the open book decomposition is of genus one with precisely $4g+4$ boundary components. Note that, Massot obtains his open book decomposition by  ``convexifying"  any given Morse function on $\S$,  whereas Johns and Ishikawa obtains their Lefschetz fibrations by ``complexifying" the Morse function at hand.

Although the existence of a genus one open book decomposition adapted to $(ST^*\S, \xi_{can})$ essentially follows from the fundamental work of Giroux \cite{g}, we believe that by providing an explicit positive factorization of the monodromy of a genus one  open book decomposition adapted to $\xi_{can}$, we fill a gap in the literature.  Moreover, we hope that our work maybe used towards settling Wendl's conjecture \cite[Conjecture 9.23]{w}: {\em Any exact filling of $(ST^*\S, \xi_{can})$ is Liouville deformation equivalent to $(DT^*\S,  \omega_{can})$.} This conjecture is true for $g=1$ as shown by Wendl \cite{w1} and some evidence was obtained recently to verify this conjecture affirmatively  by  Li, Mak and Yasui \cite{lmy} and also by Sivek and Van-Horn Morris \cite{svhm}, who  showed that any exact filling must,  at least topologically, bear some resemblance to $DT^*\S$ for $g \geq 2$.

\section{Canonical contact structure on the unit cotangent bundle of an orientable  surface}

 Suppose that $\S$ is any closed, connected, orientable surface. A cooriented contact element of $\S$ is a pair $(p, L)$ where $p \in \S$ and $L$ is a cooriented line tangent to $\S$ at $p$.  The space of cooriented contact elements of $\S$ is the collection of all cooriented contact elements of $\S$.  There exists a  canonical coorientable contact structure $\xi_{can}$ on the space of cooriented contact elements of $\S$,   which is defined as follows.  Let $\pi$ denote the natural projection of the  space of cooriented contact elements of $\S$ onto  $\S$.
For a  point $p \in S$ and a cooriented line $L$ in $T_p S$, let
$\xi_{(p,L)}$ denote the cooriented
  plane described uniquely by the
  equation $\pi_*(\xi_{(p,L)}) = L \in T_p \S$. The canonical contact structure $\xi_{can} $ on the space of cooriented contact elements of  $\S$ consists of these planes.

 If $\S$ is equipped with an arbitrary  Riemannian
        metric $\mu$,  then there is a bundle isomorphism $\Phi$ from the tangent bundle $T\S$ to the cotangent bundle $T^*\S$, which is defined fiberwise $$ \Phi_p : T_p\S \to T_p^*\S $$  by
        $ v \to \mu_p (v, -)$, for any $p \in \S$. This induces a bundle metric $\mu^*$ on $T^*\S$ by $$\mu_p^*(u_1, u_2) = \mu_p (\Phi_p^{-1} (u_1),  \Phi_p^{-1} (u_2))$$ for any $u_1 , u_2 \in T_p^*\S$. Therefore, one can define the unit tangent bundle $ST\S$ and the unit cotangent bundle $ST^*\S$ fiberwise as the collection of unit length vectors and covectors in $T_p\S$ and $T_p^*\S$, respectively.

The  space of cooriented contact elements of $\S$ can be identified with the unit tangent bundle $ST\S$ as well as the unit cotangent
        bundle $ST^*\S$.  The identification with $ST\S$ is given by taking a cooriented contact element $(p, L)$  to  the vector $v \in ST_p\S$, which is positively orthonormal  to $L$ in $T_p \S$ with respect to $\mu_p$.  Similarly, the identification with $ST^*\S$ is given by taking a cooriented contact element $(p, L)$  to  the unit covector $\Phi_p (v)= \mu_p (v, -) \in T_p^*\S$. Note that the Liouville $1$-form
        $\lambda_{can}$ on $T^*\S$ descends to  a contact $1$-form on $ST^*\S$, denoted again by $\lambda_{can}$.
        The canonical contact structure  $\xi_{can}$ on $ST^*\S$  is given by the kernel of $\lambda_{can}$ under the above identification (see for example \cite[page 32]{ge}).

        The disk tangent bundle $DT\S$ and the disk cotangent bundle $DT^*\S$ are defined fiberwise as the collection of vectors and covectors of length less than or equal to one in $T_p\S$ and $T_p^*\S$, respectively. It follows that $\partial (DT\S)= ST\S$ and $\partial (DT^*\S) = ST^*\S$. By the discussion above, we get an {\em orientation-reversing}  diffeomorphism between $DT^*\S$ and $DT\S$. Moreover, the disk
        cotangent bundle $DT^*\S$ equipped with its canonical symplectic structure $\omega_{can}=d\lambda_{can}$ is an exact  symplectic
        filling of its contact boundary  $(ST^*\S, \xi_{can}) $.

\section{A genus one Lefschetz fibration on the disk cotangent bundle of an orientable surface} \label{sec: cotlef}

\subsection{Exact symplectic Lefschetz fibrations on cotangent bundles} \label{subsec: exact}  An exact
symplectic structure on a smooth $4$-manifold $X$ with codimension $2$ corners is an exact symplectic
    $2$-form $\omega= d\lambda$ on $X$ such that the Liouville vector field (which is by definition $\omega$-dual to $\lambda$) is
    transverse to each boundary stratum of codimension $1$ and points outwards.
  Note  that $\lambda$ induces a contact form on each boundary
    stratum and $(W, \omega)$ becomes an exact symplectic filling of its
contact  boundary $(\partial X, \ker \lambda)$ provided that the corners of $W$ are rounded off (cf. \cite[Lemma 7.6]{s}).

\begin{definition} \label{def:exactsym} Suppose that $(X, \omega=d\lambda)$ is an exact symplectic $4$-manifold
 with codimension $2$ corners. We say that a smooth map $\pi : X  \to D^2$ is  an
 exact symplectic Lefschetz fibration on $(X, \omega)$ if it the following four conditions are satisfied.

(1)  There are  finitely many critical points $q_1, \ldots, q_k$ of the map $\pi$, all of which belong to the interior of $X$.  In addition, around each critical point,
the smooth map $\pi$ is modeled on the map $(z_1,
z_2) \to z_1^2 + z_2^2$ in complex local coordinates compatible with
the orientations on $X$ and $D^2$, respectively.

(2) Each fiber of the map $\pi|_{X\setminus \{q_1,
\ldots, q_k\}} : X \setminus \{q_1, \ldots, q_k\} \to D^2$ is a
symplectic submanifold of $(X, \omega)$.

(3) The boundary $\partial X$ is the union of two smooth strata meeting at a codimension $2$
corner. More precisely, $\partial X = \partial_vX \cup \partial_hX $  where $$\partial_vX = \pi^{-1} (\partial D^2)\; \mbox{and}\;
\partial_hX = \bigcup_{x \in D^2} \partial (\pi^{-1} (x) ).$$

(4) We also require that
$\pi|_{\partial_vX} :
\partial_vX \to
\partial D^2$ is smooth fibration and $\pi$ is a trivial smooth fibration over $D^2$ near
$\partial_hX$. \end{definition}

The stratum  $\partial_vX$ admits a surface fibration over $S^1$, while the stratum  $\partial_hX$ is a disjoint union
of $m > 0$ copies of the solid torus $S^1 \times D^2$, and these strata meet meet each other at the corner
$$\partial_vX \cap \partial_hX=
\partial(\partial_hX)= \coprod_{i=1}^{m} (S^1 \times \partial D^2).$$ We conclude that $\partial X$, provided that the corners of $X$ are
rounded off,  acquires
an open book decomposition given by $\pi|_{\partial X \setminus B} :
\partial X \setminus B \to \partial D^2$ where $\partial_hX$ is
viewed as a tubular neighborhood of the binding $$B = \coprod_{i=1}^{m}  (S^1
\times \{0\}).$$  Moreover, the $1$-form $\lambda$ restricts to a contact form on
$\partial X$ whose kernel is a contact structure supported by this
open book decomposition.

\begin{remark}  If we suppress the symplectic form $\omega$ on $X$,  then a smooth map  $\pi : X  \to D^2$  which satisfies only the conditions (1) and (4) in Definition~\ref{def:exactsym} is called a Lefschetz fibration on $X$.
\end{remark}

\begin{definition} \label{def:confor} For $i=1,2$, let $(X_i , \omega_i=d \lambda_i)$ be an exact symplectic $4$-manifold with convex boundary.  A conformal exact
symplectomorphism from $(X_1,
\omega_1)$ to  $(X_2, \omega_2)$ is a
diffeomorphism $\psi : X_1 \to X_2$ such that $\psi^*\lambda_2 =
C\lambda_1 + dh$ for some smooth function $h: X_1 \to \mathbb{R}$,
and some real number $C>0$. If $C=1$, then $\psi$ is called an exact
symplectomorphism.
\end{definition}

Now, we briefly recall the main result of  Johns in \cite{jo}, which we tailor appropriately to fit in with the context of the present article.  Let $\S$ be a closed, connected, orientable surface of genus $g$ equipped with a Morse function $f: \S \to \R$  and a Riemannian metric $\mu$ such that $(f, \mu)$ is Morse-Smale.  Based on this data,  Johns constructed an exact
symplectic Lefschetz fibration $\pi_f: (E, \omega) \rightarrow D^{2}$
such that there is a conformally exact symplectomorphism $\phi : (E, \omega) \to (DT^*\S,  \omega_{can})$.  Here we assume that the corners of $E$ are rounded off, so that it can be viewed as an exact symplectic $4$-manifold with convex boundary.

%\begin{remark} Definition~\ref{def:confor}  also applies to maps between exact symplectic $4$-manifolds with codimension $2$ corners. \end{remark}

\begin{remark} \label{rem: main} Since any conformal exact
symplectomorphism is, in particular, a diffeomorphism by definition, we conclude that $E$ is diffeomorphic to $DT^*\S$, and we view $\pi_f$ as a  Lefschetz fibration on $DT^*\S$.  Moreover, since $(E, \omega=d \lambda)$ is a exact symplectic filling of its contact boundary $(\partial E, \ker \lambda) $ and the conformal exact symplectomorphism $\phi : (E, \omega) \to (DT^*\S,  \omega_{can})$  satisfies $\phi^*\omega_{can} = K \omega$  for some constant $ K > 0$, the open book decomposition induced by $\pi_f$ on the boundary  $\partial(DT^*\S)=ST^*\S$ supports the canonical contact structure $\xi_{can}$. \end{remark}

%$( ST^*\S, \xi_{can}= \ker \lambda_{can})$,

In the following,  we apply Johns' method to construct an explicit  Lefschetz fibration $DT^*\S \to D^2$, with the caveat above, where the regular fiber is a surface of genus one with $4g+4$ boundary components.  To illustrate the method of construction, we first give in Section~\ref{subsec: genone} a  detailed treatment when $\S$ is a closed, orientable surface of genus one.

\subsection{The case of $T^2$} \label{subsec: genone}

The construction of Johns \cite{jo} starts with a Morse function $T^2 \to \mathbb{R}$ or equivalently a handle decomposition of $T^2$,  which possibly includes twisted $1$-handles.   Instead of giving an explicit Morse function on $T^2$, here we describe a handle decomposition of $T^2$ given by two $0$-handles, four twisted $1$-handles, and two $2$-handles. The four $1$-handles are attached to the $0$-handles as shown in Figure~\ref{fig: handledecom}. The result is an orientable surface of genus one with two boundary components. We obtain $T^2$ by attaching the $2$-handles.
\begin{figure}[h]
  \relabelbox \small {\epsfxsize=2.5in
  \centerline{\epsfbox{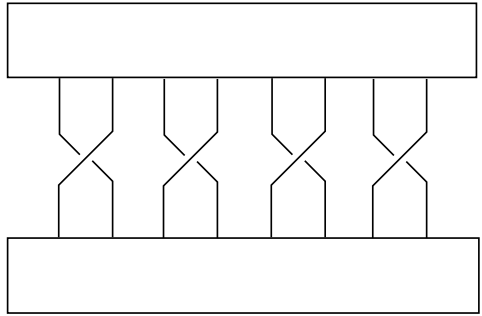}}}
  \endrelabelbox
\caption{The top and bottom rectangular regions are the $0$-handles, while each twisted band connecting them is a twisted $1$-handle.}
\label{fig: handledecom}
\end{figure}

Based on  this handle decomposition of $T^2$,
we describe a Lefschetz fibration $DT^*T^2  \to D^2$ of genus one, using the recipe given explicitly  in \cite[Section 4.3]{jo}. We first describe the regular fiber of  the Lefschetz fibration $DT^*T^2 \to D^2$, as an abstract surface of genus one with eight boundary components.   We start with two disjoint annuli $A_1$ and $A_2$ corresponding to the $0$-handles of $T^2$ as depicted
in Figure~\ref{fig: vanishingab}, where the horizontal rectangular region labeled by $A_i$  represents the annulus $A_i$ by identifying its left-edge with its right-edge, for $i=1,2$.  Now, for each twisted $1$-handle of $T^2$, we attach two $1$-handles to the union $A_1 \cup A_2$, which is equivalent to plumbing an annulus.  Therefore, we plumb four disjoint annuli $B_1, B_2, B_3, B_4$ to $A_1 \cup A_2$,  where each annulus $B_j$ is represented by the vertical rectangular region labeled by $B_j$  in Figure~\ref{fig: vanishingab}, by identifying its top-edge with its bottom-edge, for $j=1,2,3,4$.   Note that these six annuli are plumbed together in the eight overlapping squares in  Figure~\ref{fig: vanishingab}.   As a result, we obtain an  orientable surface of genus one with eight boundary components, which is the regular fiber of the Lefschetz fibration $DT^*T^2 \to D^2$.

\begin{figure}[h]
\begin{center}

\relabelbox \small {\epsfxsize=3in
  \centerline{\epsfbox{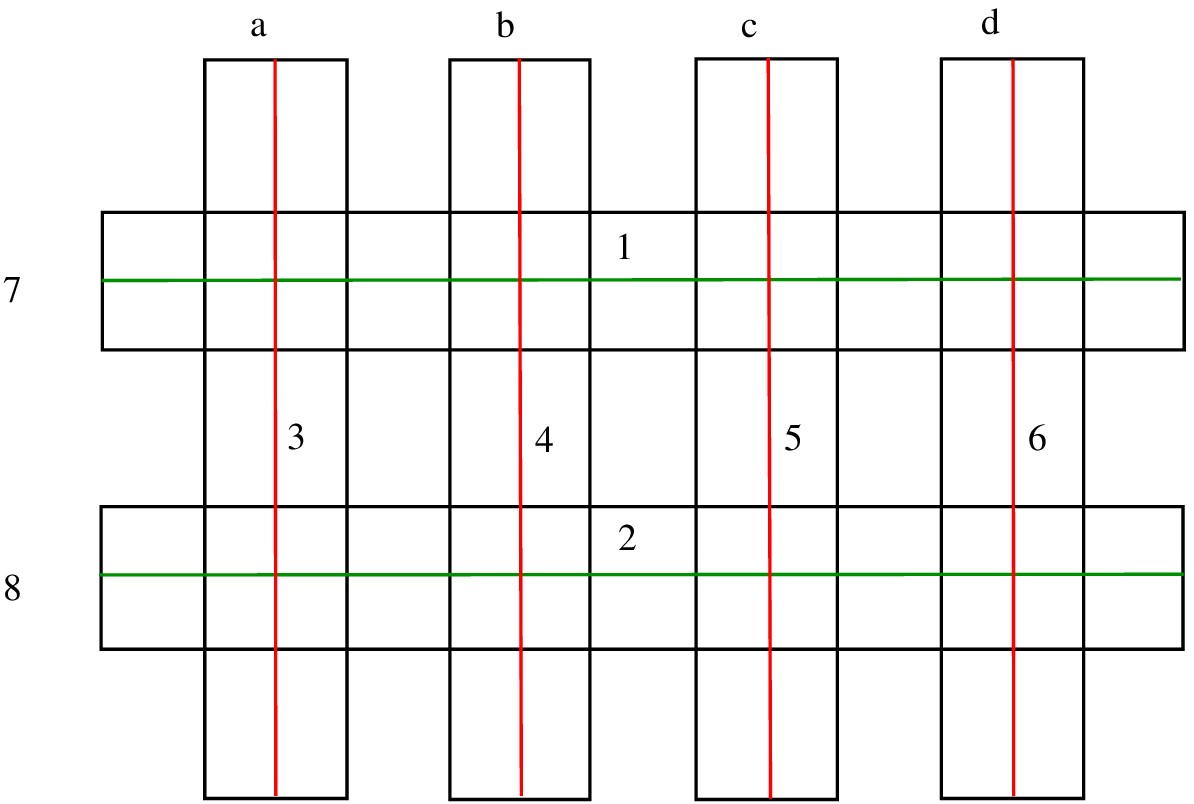}}}
   \relabel{7}{{$A_1$}}
  \relabel{8}{{$A_2$}}
\relabel{a}{{$B_1$}}
  \relabel{b}{{$B_2$}}
\relabel{c}{{$B_3$}}
  \relabel{d}{{$B_4$}}
  \relabel{1}{{$a_1$}}
  \relabel{2}{{$a_2$}}
\relabel{3}{{$b_1$}}
  \relabel{4}{{$b_2$}}
\relabel{5}{{$b_3$}}
  \relabel{6}{{$b_4$}}
  \endrelabelbox
                        \caption{The vanishing cycles $a_1, a_2, b_1, b_2, b_3, b_4$.}
                        \label{fig: vanishingab}
                    \end{center}
    \end{figure}

Next, we describe the vanishing cycles of this Lefschetz fibration. The first two vanishing cycles are the core circles of $A_1$ and $A_2$, which we denote by $a_1$ and $a_2$, respectively. The next four vanishing cycles are the core circles of $B_1, B_2, B_3, B_4$, which we denote by $b_1, b_2, b_3, b_4$, respectively (see Figure~\ref{fig: vanishingab}).   The last two vanishing cycles $c_1$ and $c_2$ are obtained by performing {\em simultaneous surgery}   of $a_1 \cup  a_2$ and $b_1  \cup b_2 \cup b_3  \cup b_4$ at each point where these curves meet. This means that any time $a_i$ intersects some $b_j$, the intersection point is resolved by a surgery, where $a_i$ turns to the left as illustrated in Figure~\ref{fig: lagsur}. By resolving the eight intersection points in Figure~\ref{fig: vanishingab}, we obtain a curve with two components, denoted by $c_1$ and $c_2$ as shown in Figure~\ref{fig: vanishingc}.

 \begin{figure}[h]
  \relabelbox \small {\epsfxsize=3in
  \centerline{\epsfbox{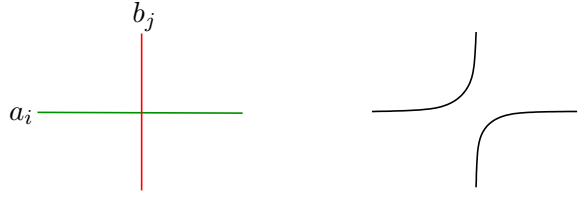}}}
  \relabel{1}{{$a_i$}}
  \relabel{2}{{$b_j$}}
  \endrelabelbox
\caption{Surgery at a point where $a_i$ meets $b_j$.}
\label{fig: lagsur}
\end{figure}

\begin{figure}[h]
                    \begin{center}

                         \relabelbox \small {\epsfxsize=3in
  \centerline{\epsfbox{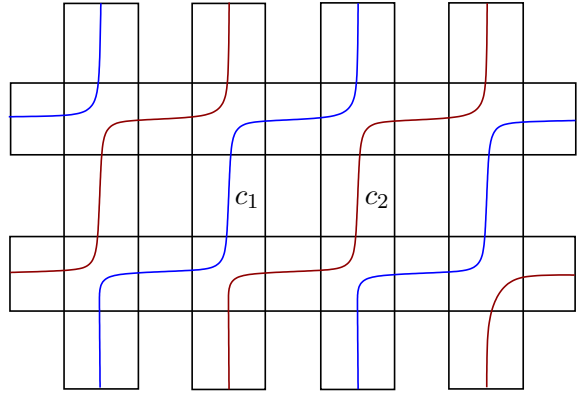}}}
  \relabel{1}{{$c_1$}}
  \relabel{2}{{$c_2$}} \endrelabelbox
                        \caption{The last two vanishing cycles $c_1$ and $c_2$. }
                        \label{fig: vanishingc}
                    \end{center}
    \end{figure}

The monodromy of the Lefschetz fibration $DT^*T^2 \to D^2$ is given by the composition of positive Dehn twists $$ D(a_1) D(a_2) D(b_1) D(b_2) D(b_3) D(b_4) D(c_1) D(c_2).$$

To summarize, we proved the following result.

\begin{proposition} \label{prop: torus} There exists a Lefschetz fibration $DT^*T^2 \cong T^2 \times D^2 \to D^2$, whose regular fiber is a surface of genus one with eight boundary components. The monodromy of this Lefschetz fibration is given by the composition of positive Dehn twists $$ D(a_1) D(a_2) D(b_1) D(b_2) D(b_3) D(b_4) D(c_1) D(c_2)$$ where the curves $a_1, a_2, b_1, b_2, b_3, b_4, c_1, c_2$ are depicted in Figures~\ref{fig: genusoneab} and \ref{fig: genusonec} on a standard genus one surface with eight boundary components.
\end{proposition}

\begin{figure}[h]
                    \begin{center}
                         \relabelbox \small {\epsfxsize=4.5in
  \centerline{\epsfbox{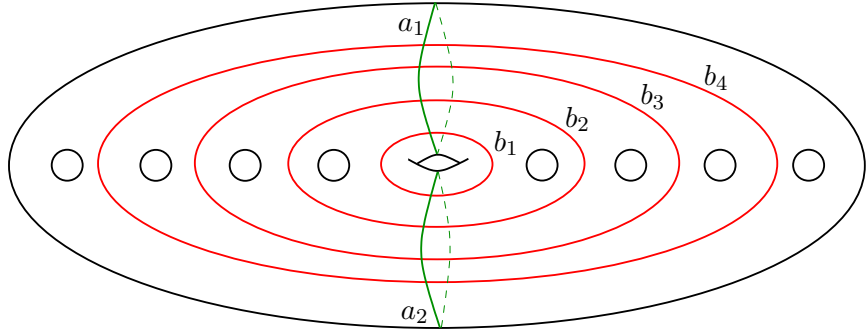}}}
  \relabel{1}{{$a_1$}}
  \relabel{2}{{$a_2$}}
\relabel{3}{{$b_1$}}
  \relabel{4}{{$b_2$}}
\relabel{5}{{$b_3$}}
  \relabel{6}{{$b_4$}}
  \endrelabelbox
                        \caption{The vanishing cycles $a_1, a_2, b_1, b_2, b_3, b_4$. }
                        \label{fig: genusoneab}
                    \end{center}
    \end{figure}

\begin{figure}[h]
                    \begin{center}
                         \relabelbox \small {\epsfxsize=4.5in
  \centerline{\epsfbox{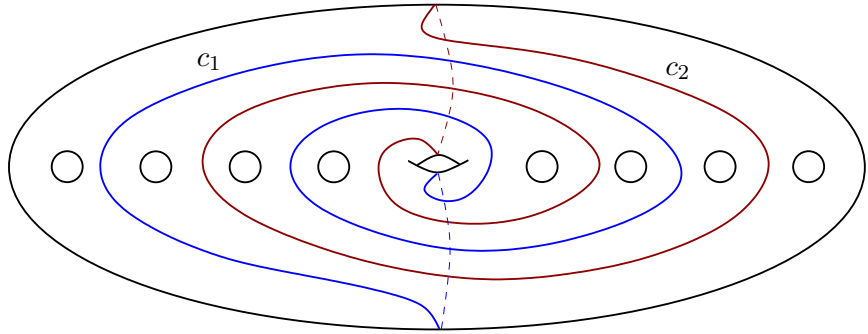}}}
  \relabel{1}{{$c_1$}}
  \relabel{2}{{$c_2$}}
 \endrelabelbox
                        \caption{The vanishing cycles $c_1$ and $c_2$.}
                        \label{fig: genusonec}
                    \end{center}
    \end{figure}

\subsection{General case} \label{subsec: geng}

The discussion in Section~\ref{subsec: genone} can be generalized to the case of an arbitrary closed, connected, orientable surface $\S$ of genus $g \geq 0$ as follows. There is a handle decomposition of $\S$ given by two $0$-handles, $2g+2$  twisted $1$-handles, and two $2$-handles, where we attach the $1$-handles to the $0$-handles analogous to the $g=1$ case (see Figure~\ref{fig: handledecom}).  The result is an orientable surface of genus $g$ with two boundary components. By attaching the $2$-handles, we obtain $\S$.

In the following, we describe a Lefschetz fibration $DT^*\S \to D^2$ of genus one, based on this handle decomposition of $\S$.
The regular fiber of  this Lefschetz fibration is obtained by plumbing $2g+2$ disjoint annuli $B_1, \ldots, B_{2g+2}$ corresponding to the $2g+2$  twisted $1$-handles, to the two disjoint annuli $A_1$ and $A_2$ corresponding to the $0$-handles, analogous to the $g=1$ case (see Figure~\ref{fig: vanishingab}).  The result is an  orientable surface of genus one with $4g+4$ boundary components.

The vanishing cycles of this Lefschetz fibration are obtained as follows.  Let $a_i$ denote the core circle of $A_i$ for  $i=1,2$ and $b_j$ denote the core circle of $B_j$ for $j=1, \ldots, 2g+2$. Let $c_1$ and $c_2$
denote the two curves on the fiber obtained by simultaneous  surgery of $a_1 \cup a_2$ and $b_1 \cup b_2 \cup  \cdots \cup b_{2g+2}$, at each point where they meet, analogous to the $g=1$ case (see Figure~\ref{fig: vanishingc}).  Therefore, we have the following result.

\begin{theorem} \label{thm: gencase} For any integer $g \geq 0$,  the disk cotangent bundle $DT^*\S$ admits a Lefschetz fibration over $D^2$,  whose regular fiber is a surface of genus one with $4g+4$ boundary components. The monodromy of this Lefschetz fibration is given by the composition of positive Dehn twists $$ D(a_1) D(a_2) D(b_1) D(b_2) \cdots  D(b_{2g+2})  D(c_1) D(c_2)$$ where the curves $a_1, a_2, b_1, b_2, \ldots , b_{2g+2}, c_1, c_2$  are depicted in Figures~\ref{fig: genusgab} and \ref{fig: genusgc} on a standard genus one surface with $4g+4$ boundary components.
\end{theorem}

The next corollary immediately follows from Theorem~\ref{thm: gencase} coupled with Remark~\ref{rem: main}.

\begin{corollary} \label{cor: gen} For any integer $g \geq 0$, the unit cotangent bundle $ST^*\S$ admits an open book decomposition adapted to the
canonical contact structure $\xi_{can}$, whose page is a genus one surface with $4g+4$ boundary components. The monodromy of this open book decomposition is given by the composition of positive Dehn twists $$ D(a_1) D(a_2) D(b_1) D(b_2) \cdots  D(b_{2g+2})  D(c_1) D(c_2)$$ where the curves $a_1, a_2, b_1, b_2, \ldots , b_{2g+2}, c_1, c_2$  are depicted in Figures~\ref{fig: genusgab} and \ref{fig: genusgc}.
\end{corollary}

\begin{figure}[h]
                    \begin{center}
                         \relabelbox \small {\epsfxsize=5in
  \centerline{\epsfbox{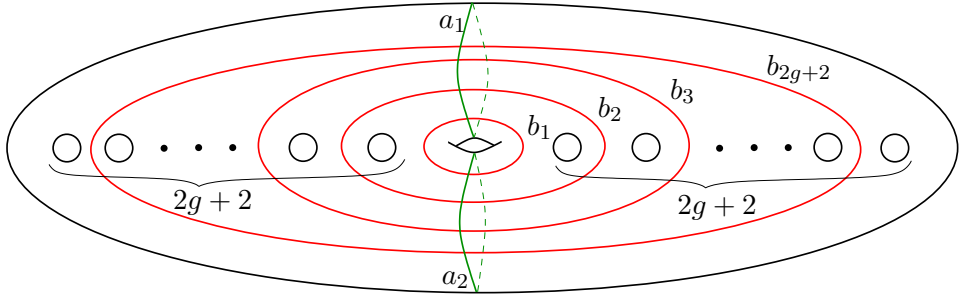}}}
  \relabel{1}{{$a_1$}}
  \relabel{2}{{$a_2$}}
\relabel{3}{{$b_1$}}
  \relabel{4}{{$b_2$}}
\relabel{5}{{$b_3$}}
  \relabel{6}{{$b_{2g+2}$}}
  \relabel{7}{{$2g+2$}}
  \relabel{8}{{$2g+2$}}
   \endrelabelbox
                        \caption{The vanishing cycles $a_1, a_2, b_1, b_2, b_3, \ldots, b_{2g+2}$. }
                        \label{fig: genusgab}
                    \end{center}
    \end{figure}

  \begin{figure}[h]
                    \begin{center}
                         \relabelbox \small {\epsfxsize=5in
  \centerline{\epsfbox{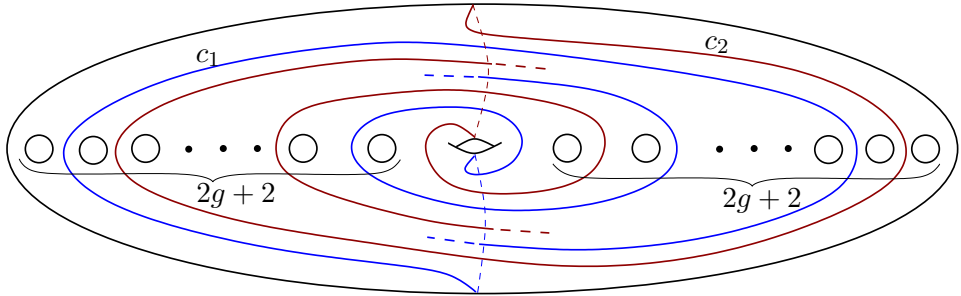}}}
  \relabel{1}{{$c_1$}}
  \relabel{2}{{$c_2$}}
  \relabel{3}{{$2g+2$}}
  \relabel{4}{{$2g+2$}}
   \endrelabelbox
                        \caption{The vanishing cycles $c_1$ and $c_2$. }
                        \label{fig: genusgc}
                    \end{center}
    \end{figure}

 Note that for $g=0,1$,  the results above are known to be not optimum in the following sense. For $g=0$, there is a {\em planar} Lefschetz fibration  $DT^*S^2 \to D^2$, whose regular fiber is the annulus and whose monodromy is the square of the positive Dehn twist along the core circle. The restriction of this Lefschetz fibration to the boundary gives a {\em planar}  open book decomposition of $ ST^*S^2 \cong \mathbb{RP}^3$ adapted to $\xi_{can}$. In other words, the support genus of $(\mathbb{RP}^3, \xi_{can})$ is zero, while its binding number is equal to two (cf. \cite{eo}).

 As we have already mentioned in the Introduction, the contact $3$-manifold $(ST^* \S, \xi_{can})$ is not planar for $g \geq 1$.  For $g=1$, Van
        Horn-Morris \cite{vhm} constructed an explicit open book  decomposition adapted to $(ST^* T^2 \cong T^3, \xi_{can})$ whose page is a genus one surface with {\em three} boundary components. Therefore, the support genus of $(T^3, \xi_{can})$ is one, while its binding number is less than or equal to three. As observed by Massot \cite{m}, and confirmed via different methods by our Corollary~\ref{cor: gen}, for any $g \geq 1$, the support genus of $(ST^* \S, \xi_{can})$ is one and the binding number is less than or equal to $4g+4$. Therefore the following question appears naturally:

 \begin{question} What is the binding number of $(ST^* \S, \xi_{can})$ for $g \geq 1$? \end{question}

Note that  $DT^*T^2 \cong T^2 \times D^2$  does not  admit a planar Lefschetz fibration over $D^2$, even if we do not impose any boundary conditions. Suppose, otherwise,  that $T^2 \times D^2$ admits a planar Lefschetz fibration over $D^2$. Then we would have a planar strongly symplectically fillable contact structure on the boundary  $\del (T^2 \times D^2) \cong T^3$. This gives a contradiction since $\xi_{can}$  is the  unique  strongly symplectically  fillable contact structure on $T^3$ by Eliashberg \cite{el}, and it is non-planar \cite{et}.

  %Moreover, according to Wendl \cite{w}, any minimal strong symplectic filling of $(T^3, \xi_{can})$ is symplectic deformation equivalent to $(DT^*T^2,  \omega_{can})$. Thereforenon-planar \cite{et} and it ,  $DT^*T^2$   admits a genus one Lefschetz fibration over $D^2$ filling the Van Horn-Morris' open book.

\section{Another genus one Lefschetz fibration on the disk cotangent bundle of an orientable surface} \label{sec: tanlef}

\subsection{Lefschetz fibrations on disk cotangent bundles} \label{sec: ish}

%Let $\S$  denote a closed, connected, orientable surface of genus $g \geq 0$ . Any Riemannian metric on $\S$ induces a diffeomorphism between the disk cotangent bundle $DT^*\S$ and the disk tangent bundle $DT\S$. Therefore,  Theorem~\ref{thm: gencase} indeed implies the existence of a Lefschetz fibration $DT\S  \to D^2$ with $2g+6$  vanishing cycles, whose regular fiber is a genus one surface with $4g+4$ boundary components.

In \cite[Proposition 3.1]{i}, Ishikawa constructs a Lefschetz fibration on the disk {\em tangent}  bundle of an orientable surface $\S$, based on the choice of a Morse function on $\S$. His construction is based on Lemma 3.2, which in turn relies on Lemma 2.6 in his paper. We would like to point out that there is an orientation error in Lemma 2.6 of Ishikawa's paper.  His choice of complex charts in Lemma 2.6 is {\em orientation-reversing} for the tangent bundle rather than orientation-preserving. Therefore, the Lefschetz fibration he constructs in Proposition 3.1 on the disk tangent bundle $DT\S$ is  {\em{\bf achiral}}, i.e., all Dehn twists are left-handed. By reversing the orientation of the total space we get a Lefschetz fibration on the disk {\em cotangent}  bundle $DT^*\S$. Thus, we conclude that Ishikawa in fact constructs Lefschetz fibrations on $DT^*\S$ rather than $DT\S$.
\begin{remark} Here is another way to see the error in Ishikawa'a paper \cite{i}. For $\S= S^2$, Ishikawa's method would give a Lefschetz fibration on the disk tangent bundle of $S^2$, which is the $D^2$-bundle over $S^2$ with Euler number $+2$. This would imply that the $D^2$-bundle over $S^2$ with Euler number $+2$ admits a Stein structure (cf. \cite{ao,lp}) which, indeed,  contradicts to the adjunction inequality of Lisca and Mati\'{c} \cite{lm} for Stein surfaces.

\end{remark}

With this caveat in mind, we briefly describe Ishikawa's construction of a Lefschetz fibration on $DT^*\S$.
To construct his fibration,  Ishikawa starts with  any {\em admissible divide} $P$ on $\S$.  For the purposes of the present paper,  a divide $P \subset \S$ is a generic immersion of the disjoint union of finitely many copies of the unit circle.   A divide $P$ is called admissible if it is connected, each component of $\S \setminus P$  is simply connected and $P$ admits a checkerboard coloring, which means that one can assign black or white color to each component of $\S \setminus P$, such
that any two neighboring components are assigned opposite colors.

Based on an admissible divide $P$ on $\S$,  there is a Morse function $f_P : \S \to \mathbb{R}$ associated with $P$, which essentially means that the zero level set of $f_P$ coincides with $P$, each double point of $P$ corresponds to a critical point of $f_P$ of index one, and  each black (resp. white) region of $\S \setminus P$ contains an index two (resp. zero) critical point of $f_P$. The Morse function $f_P$,  in turn,  gives an ``almost complexified Morse function" $F_P : T^*\S \to \mathbb{C}$ which  descends to a Lefschetz fibration $\pi_P : DT^*\S \to D^2$.   Moreover, generalizing the work of A'Campo \cite{ac}, Ishikawa describes  how to obtain the regular fiber and the monodromy of the Lefschetz fibration $\pi_P$, based only on the divide $P$.

In the following, by choosing a particular admissible divide $P$ on $\S$ and applying Ishikawa's method, we obtain an explicit Lefschetz fibration $DT^*\S \to D^2$ in Theorem~\ref{thm: gene}, with $2g+6$  vanishing cycles, whose regular fiber is a genus one surface with $4g+4$ boundary components. To illustrate the method of construction, we first give in Section~\ref{subsec: one} a  detailed treatment when $\S$ is a closed, orientable surface of genus one.

%In Section~\ref{subsec: one}, we give detailed constructions and proofs of the above discussion for the case $g=1$.

\subsection{The case of $T^2$} \label{subsec: one}
In this subsection, by choosing a particular admissible divide on the torus $T^2$, we construct an explicit Lefschetz fibration $ DT^* T^2  \to D^2$ of genus one.
The admissible divide $P$ we have in mind is the union of the four curves $P_1, P_2, P_3, P_4$ on $T^2$  intersecting as in Figure~\ref{fig: divideone}.  The complement $T^2 \setminus P$ has four connected components each of which is a disk. We assign a checkerboard  coloring to $T^2 \setminus P$ as follows. The component bounded by the bold curves, facing the reader in Figure~\ref{fig: divideone}, is assigned the white color, which in turn,  determines the color of the remaining three components of $T^2 \setminus P$, since neighboring components should have opposite colors. Based on this choice of $P$,  there is a Lefschetz fibration $DT^* T^2  \to D^2$. The fiber of this Lefschetz fibration can be constructed via A'Campo's method \cite[page 15]{ac} as follows: We start with a roundabout, as depicted in Figure~\ref{fig: roundab},  for each of the four double points of the divide $P$.

\begin{figure}[h]
                    \begin{center}
                         \relabelbox \small {\epsfxsize=3.2in
  \centerline{\epsfbox{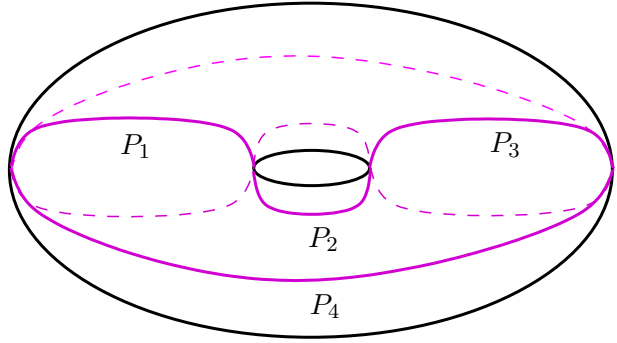}}}
  \relabel{1}{{$P_1$}}
  \relabel{2}{{$P_2$}}
  \relabel{3}{{$P_3$}}
  \relabel{4}{{$P_4$}}

   \endrelabelbox
                        \caption{The divide $P = P_1 \cup  P_2 \cup  P_3 \cup  P_4$ on $T^2$.}
                        \label{fig: divideone}
                    \end{center}
    \end{figure}

\begin{figure}[h]
                    \begin{center}
                         \relabelbox \small {\epsfxsize=2.2in
  \centerline{\epsfbox{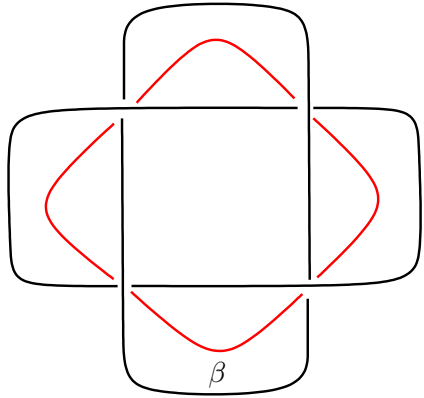}}}
 \relabel{1}{{$\b$}}
 \endrelabelbox
                        \caption{The roundabout:  an annulus embedded in $\mathbb{R}^3$, with its core circle $\b$.}
                        \label{fig: roundab}
                    \end{center}
    \end{figure}

For each edge in $P$ connecting any two double points, we insert a half-twisted band connecting the corresponding roundabouts. As a result, we get a genus one surface with eight boundary components as shown in Figure~\ref{fig: roundvanish}.

\begin{figure}[h]
                    \begin{center}
                         \relabelbox \small {\epsfxsize=4.7in
  \centerline{\epsfbox{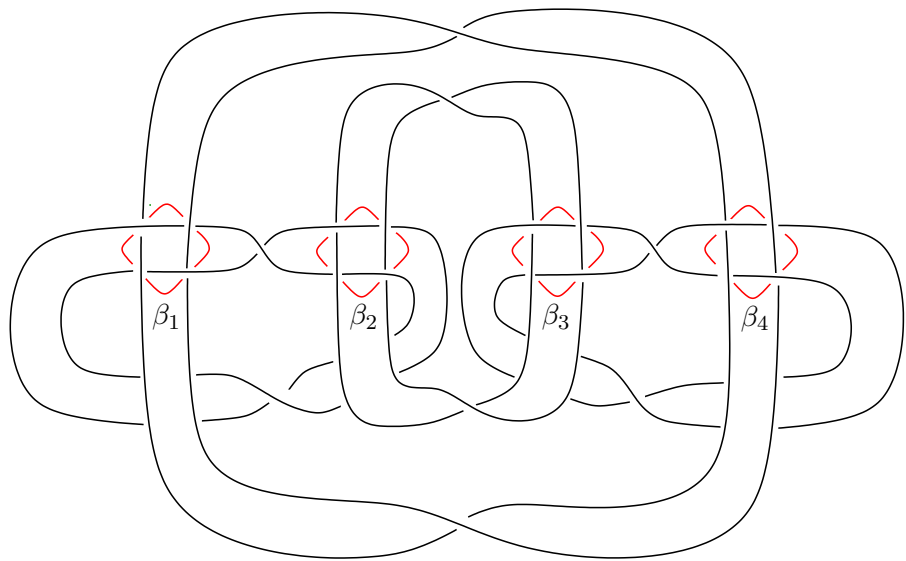}}}
 \relabel{1}{{$\b_1$}}
  \relabel{2}{{$\b_2$}}
  \relabel{3}{{$\b_3$}}
  \relabel{4}{{$\b_4$}}
 \endrelabelbox
                        \caption{The vanishing cycles $\b_1, \b_2, \b_3, \b_4$.}
                        \label{fig: roundvanish}
                    \end{center}
    \end{figure}

    \begin{figure}[h]
                    \begin{center}
                         \relabelbox \small {\epsfxsize=5in
  \centerline{\epsfbox{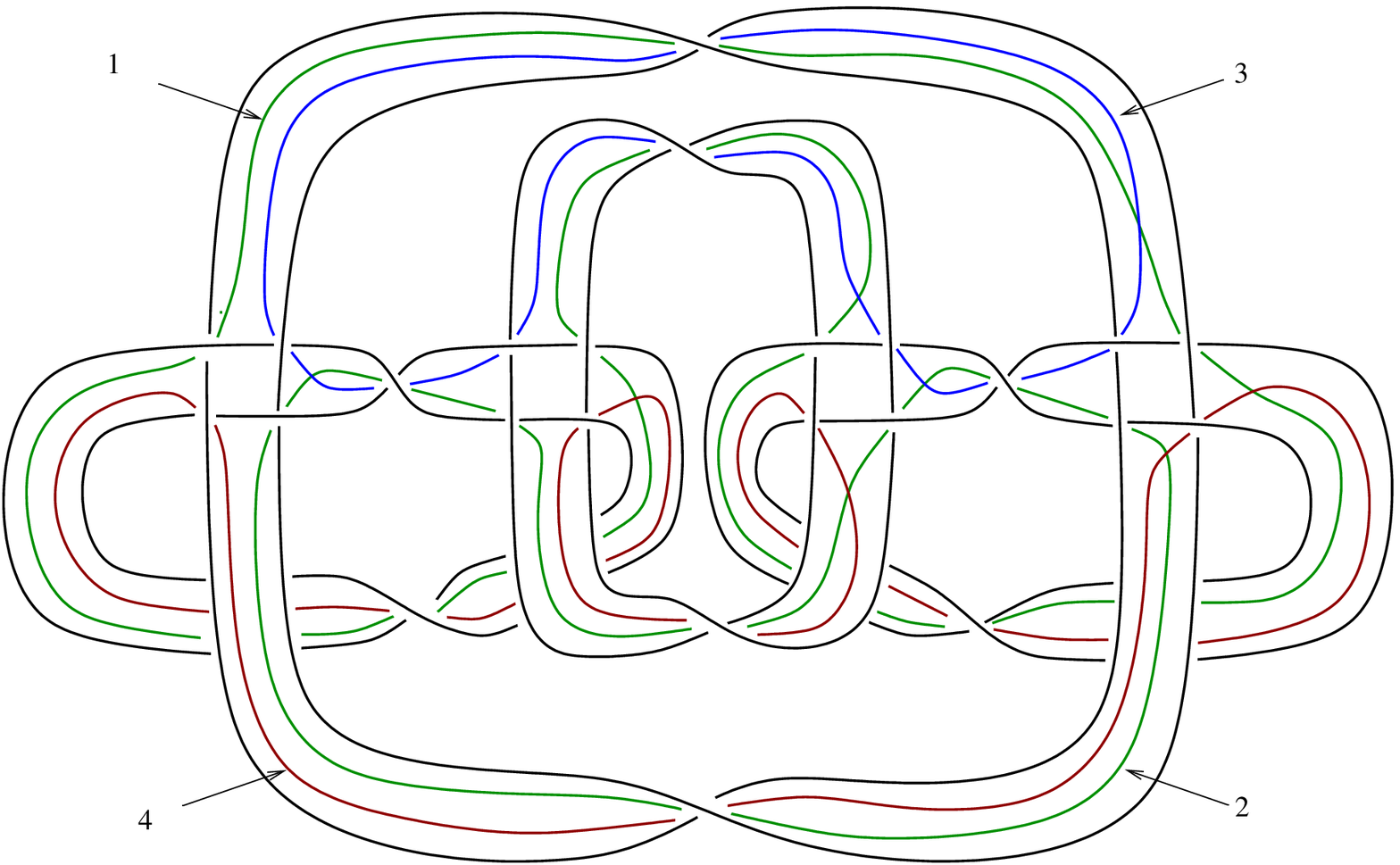}}}
\relabel{1}{{$\a_2$}}
  \relabel{2}{{$\a_1$}}
  \relabel{3}{{$\g_1$}}
  \relabel{4}{{$\g_2$}}

   \endrelabelbox
                        \caption{The vanishing cycles $\a_1, \a_2, \g_1, \g_2$.}
                        \label{fig: roundvanish2}
                    \end{center}
    \end{figure}

The monodromy of this Lefschetz fibration $DT^* T^2 \to D^2$ is given by the product of eight positive Dehn twists along the curves $\a_1, \a_2, \b_1, \b_2, \b_3, \b_4, \g_1, \g_2$. The curves $\b_1, \b_2, \b_3, \b_4$ are the core circles of the four roundabouts in Figure~\ref{fig: roundvanish}. The curves $\a_1$ and $\a_2$ are given as the boundary of the two white regions in the checkerboard  coloring we fixed above, while the curves $\g_1$ and $\g_2$ are  given as the boundary of the two black regions. We depicted the curves $\a_1, \a_2, \g_1, \g_2$ in Figure~\ref{fig: roundvanish2}. We summarize our discussion in Proposition~\ref{prop: tor} below.

\begin{proposition} \label{prop: tor} There exists a Lefschetz fibration $DT^* T^2  \cong T^2 \times D^2  \to D^2$, whose regular fiber is a genus one surface with eight boundary components. The monodromy of this Lefschetz fibration is given by the composition of positive Dehn twists $$ D(\a_1) D(\a_2) D(\b_1) D(\b_2) D(\b_3) D(\b_4) D(\g_1) D(\g_2)$$ where the curves $\a_1, \a_2, \b_1, \b_2, \b_3, \b_4, \g_1, \g_2$ are depicted in  Figures~\ref{fig: roundvanish} and \ref{fig: roundvanish2} on a genus one surface with eight boundary components.
\end{proposition}

%\begin{remark}  It is clear that $DT^*T^2 \cong T^2 \times D^2 \cong DT(T^2)$. There is no a priori reason, however,  for the existence of a fiberwise diffeomorphism between the Lefschetz fibration  $DT^*T^2 \to D^2$  in Proposition~\ref{prop: torus} and the Lefschetz fibration $DT(T^2) \to D^2$ in Proposition~\ref{prop: tor}, that takes the eight vanishing cycles  of the former to the eight vanishing cycles of the latter.  This is the content of our next result.  \end{remark}

\begin{remark} There is no a priori reason for the existence of a fiberwise diffeomorphism between the Lefschetz fibration $DT^*\S\to D^2$ described in Theorem~\ref{thm: gencase} and the Lefschetz fibration $DT^*\S \to D^2$ described in Theorem~\ref{thm: gene} that takes the $2g+6$  vanishing cycles  of the former to that of the latter. In Proposition~\ref{prop: isom},  we provide such a diffeomorphism for the case $g=1$.  A general statement for any $g \geq 0$ is given in Theorem~\ref{thm: isomgene} below, whose proof is analogous to the $g=1$ case.  \end{remark}

\begin{proposition} \label{prop: isom} The Lefschetz fibration $DT^* T^2   \to D^2$ described in Proposition~\ref{prop: torus} and the Lefschetz fibration $DT^* T^2   \to D^2$ described in  Proposition~\ref{prop: tor} are isomorphic.
\end{proposition}

%Any Riemannian metric on $T^2$ induces a diffeomorphism between the disk cotangent bundle $DT^*T^2$ and the disk tangent bundle $DT(T^2)$. Therefore,  Proposition~\ref{prop: torus} indeed implies the existence of a Lefschetz fibration $DT(T^2)  \to D^2$ with eight vanishing cycles, whose regular fiber is a genus one surface with eight boundary components---just as we stated in Proposition~\ref{prop: tor}.  But the point is that,

\begin{proof} Let $\pi_1$ denote the Lefschetz fibration $DT^* T^2  \to D^2$ described in Proposition~\ref{prop: torus} and $\pi_2$ denote the Lefschetz fibration $DT^* T^2   \to D^2$ described in  Proposition~\ref{prop: tor}. The regular fibers of $\pi_1$ and $\pi_2$ are clearly diffeomorphic as abstract surfaces. In the following, we establish an explicit diffeomorphism between these fibers which also preserves the corresponding vanishing cycles.

Recall that the fiber (see Figure~\ref{fig: vanishingab})  of $\pi_1$ is obtained simply by plumbing  the four disjoint annuli
$B_1 \cup  B_2 \cup B_3 \cup B_4$ onto two disjoint annuli $A_1 \cup A_2$, along the eight overlapping squares. Moreover, the core circle of each annulus is a vanishing cycle. In the following, we show that the fiber (shown in Figure~\ref{fig: roundvanish2})  of $\pi_2$  is obtained exactly in the same way.

\begin{figure}[h]
                    \begin{center}
                         \relabelbox \small {\epsfxsize=4in
  \centerline{\epsfbox{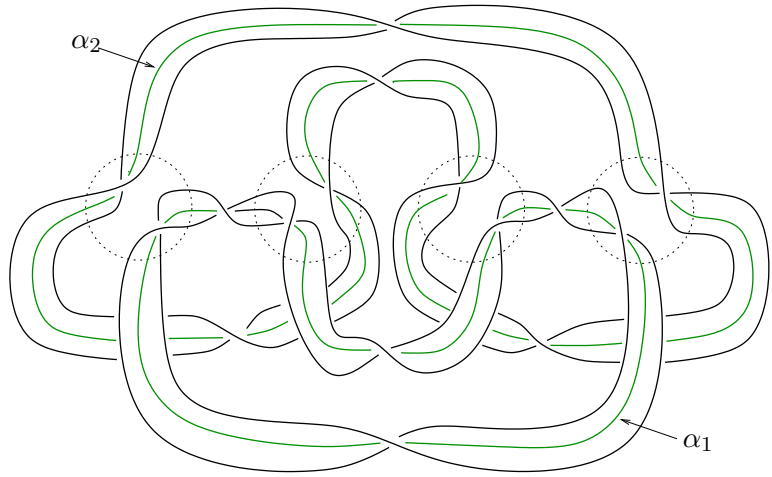}}}
\relabel{1}{{$\a_2$}}
  \relabel{2}{{$\a_1$}}

   \endrelabelbox
                        \caption{The neighborhoods of the vanishing cycles $\a_1$ and $\a_2$.}
                        \label{fig: annul}
                    \end{center}
    \end{figure}

For $i=1,2$, a neighborhood of the vanishing cycle $\a_i$ on the regular fiber of $\pi_2$, which is shown in Figure~\ref{fig: roundvanish2}, is indeed an annulus as we depicted again in Figure~\ref{fig: annul}. Note that to emphasize the neighborhoods of $\a_1$ and $\a_2$,  we erased part of the fiber in Figure~\ref{fig: roundvanish2}, which we indicated by the dotted circles in Figure~\ref{fig: annul}. Similarly, for $j=1,2,3,4$, a neighborhood of the curve $\b_j$ on the fiber in Figure~\ref{fig: roundvanish} is a roundabout. Now we claim that the fiber in Figure~\ref{fig: roundvanish} (or Figure~\ref{fig: roundvanish2}) can be obtained by plumbing the four disjoint roundabouts, each one is a neighborhood of $\b_j$, onto the disjoint union of the neighborhoods of $\a_1$ and $\a_2$ depicted in Figure~\ref{fig: annul}. To prove our claim, we illustrate in Figure~\ref{fig: round} the result of plumbing a roundabout onto the disjoint neighborhoods of $\a_1$  and $\a_2$ inside a dotted circle.

     \begin{figure}[h]
                    \begin{center}
                         \relabelbox \small {\epsfxsize=3.5in
  \centerline{\epsfbox{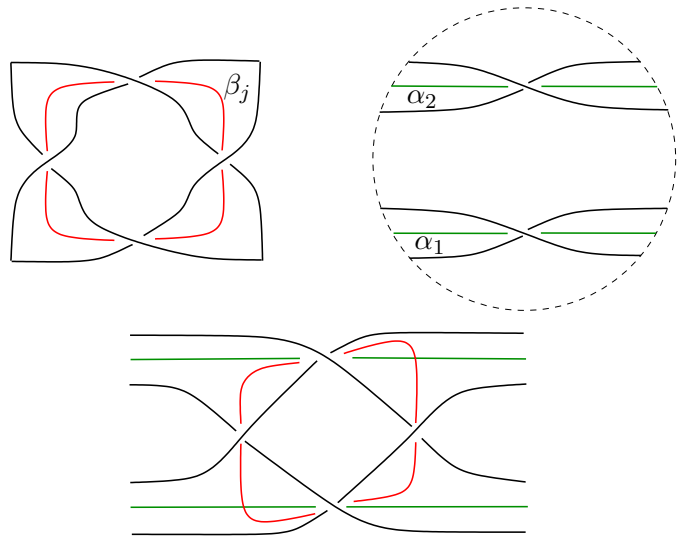}}}
\relabel{1}{{$\b_j$}}
\relabel{2}{{$\a_2$}}
  \relabel{3}{{$\a_1$}}
\endrelabelbox
                        \caption{Plumbing (bottom) the roundabout (top left) onto the annuli neighborhoods of $\a_1$ and $\a_2$ (top right).}
                        \label{fig: round}
                    \end{center}
    \end{figure}

Note that inside each one of the four dotted circles in Figure~\ref{fig: annul}, there are two disjoint ``twisted squares", and each plumbing takes place inside one of these circles. Therefore, to establish a diffeomorphism between the regular fiber of $\pi_2$  and the regular  the fiber of $\pi_1$, we identify the neighborhood of $\a_i$ with the annulus $A_i$, which is a neighborhood of $a_i$ and the neighborhood of $\b_j$ with the annulus $B_j$, which a neighborhood of $b_j$. Our discussion so far  shows that the regular fiber of $\pi_2$  is diffeomorphic to the regular fiber of $\pi_1$  by a diffeomorphism sending $\a_i$ to $a_i$ for $i=1,2$ and $\b_j$ to $b_j$ for $j=1,2,3,4$.

 Finally, recall that the last vanishing cycles $c_1$ and $c_2$ of $\pi_1$ are obtained by simultaneous  surgery of $a_1 \cup a_2$ and $b_1 \cup b_2 \cup b_3 \cup b_4$, at each point where they meet. We just observe that
 $\g_1$ and $\g_2$  are also obtained by simultaneous surgery of $\a_1 \cup \a_2$ and $\b_1 \cup \b_2 \cup \b_3 \cup \b_4$, at each point where they meet. We conclude that the diffeomorphism above takes $\g_i$ to $c_i$ as well for $i=1,2$. Therefore, there is an isomorphism between the two genus one Lefschetz fibrations  $\pi_2: DT^* T^2 \to D^2$ and $\pi_1: DT^* T^2   \to D^2$.
\end{proof}

\subsection{General case} \label{subsec: general}  The discussion in Section~\ref{subsec: one} can be generalized to the case of an arbitrary closed, connected, orientable surface $\S$ of genus $g \geq 0$ as follows. We start with the admissible divide $P$ on $\S$ given in Figure~\ref{fig: dividegen}.
\begin{figure}[h]
                    \begin{center}
                         \relabelbox \small {\epsfxsize=4.7in
  \centerline{\epsfbox{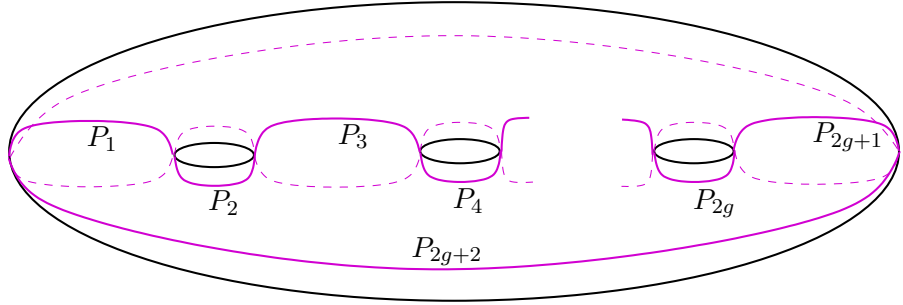}}}
  \relabel{1}{{$P_1$}}
  \relabel{2}{{$P_2$}}
  \relabel{3}{{$P_3$}}
\relabel{4}{{$P_4$}}
   \relabel{5}{{$P_{2g}$}}
   \relabel{6}{{$P_{2g+1}$}}
   \relabel{7}{{$P_{2g+2}$}}

  \endrelabelbox
                        \caption{The divide $P = P_1 \cup  P_2 \cup  \cdots  \cup  P_{2g+2}$ on $\S$.}
                        \label{fig: dividegen}
                    \end{center}
    \end{figure}

Just as in the case of genus one, $\S \setminus P$ has four connected components, each of which is a disk and $\S \setminus P$ admits a checkerboard coloring.  Based on this choice of $P$,  there is a Lefschetz fibration $DT^* \S  \to D^2$. The fiber of this Lefschetz fibration can be constructed  as described in Section~\ref{subsec: one}:  We start with a roundabout for each double point of the divide $P$, and for each edge in $P$ connecting any two double points, we insert a half-twisted band connecting the corresponding roundabouts. As a result, we get a genus one surface with $4g+4$ boundary components, as shown in Figure~\ref{fig: roundvanishdeneme}.

  \begin{figure}[h]
                    \begin{center}
                         \relabelbox \small {\epsfxsize=6in
  \centerline{\epsfbox{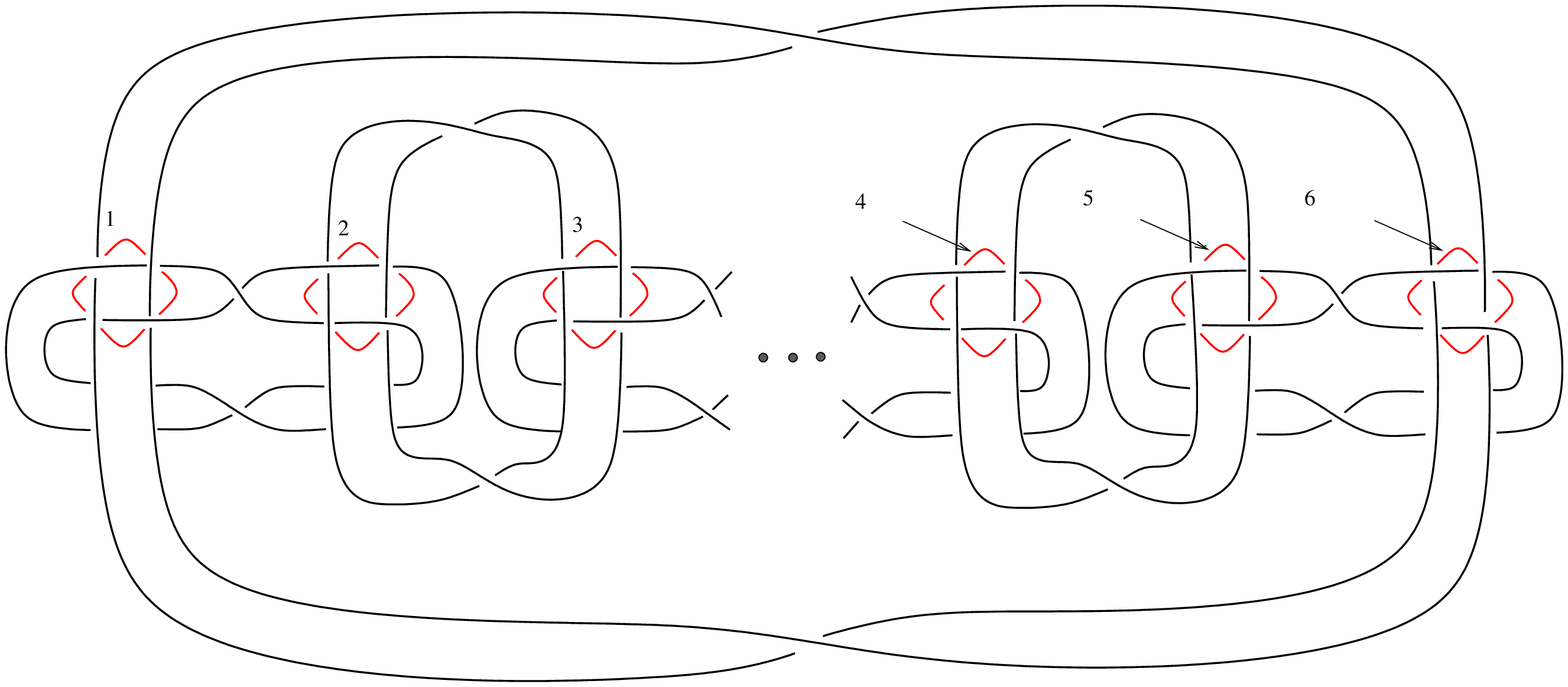}}}
\relabel{1}{{$\b_1$}}
  \relabel{2}{{$\b_2$}}
  \relabel{3}{{$\b_3$}}
  \relabel{4}{{$\b_{2g}$}}
  \relabel{5}{{$\b_{2g+1}$}}
  \relabel{6}{{$\b_{2g+2}$}}
   \endrelabelbox
                        \caption{The vanishing cycles $\ \b_1,  \b_2, \ldots, \b_{2g+2}$.}
                        \label{fig: roundvanishdeneme}
                    \end{center}
    \end{figure}

  \begin{figure}[h]
                    \begin{center}
                         \relabelbox \small {\epsfxsize=6in
  \centerline{\epsfbox{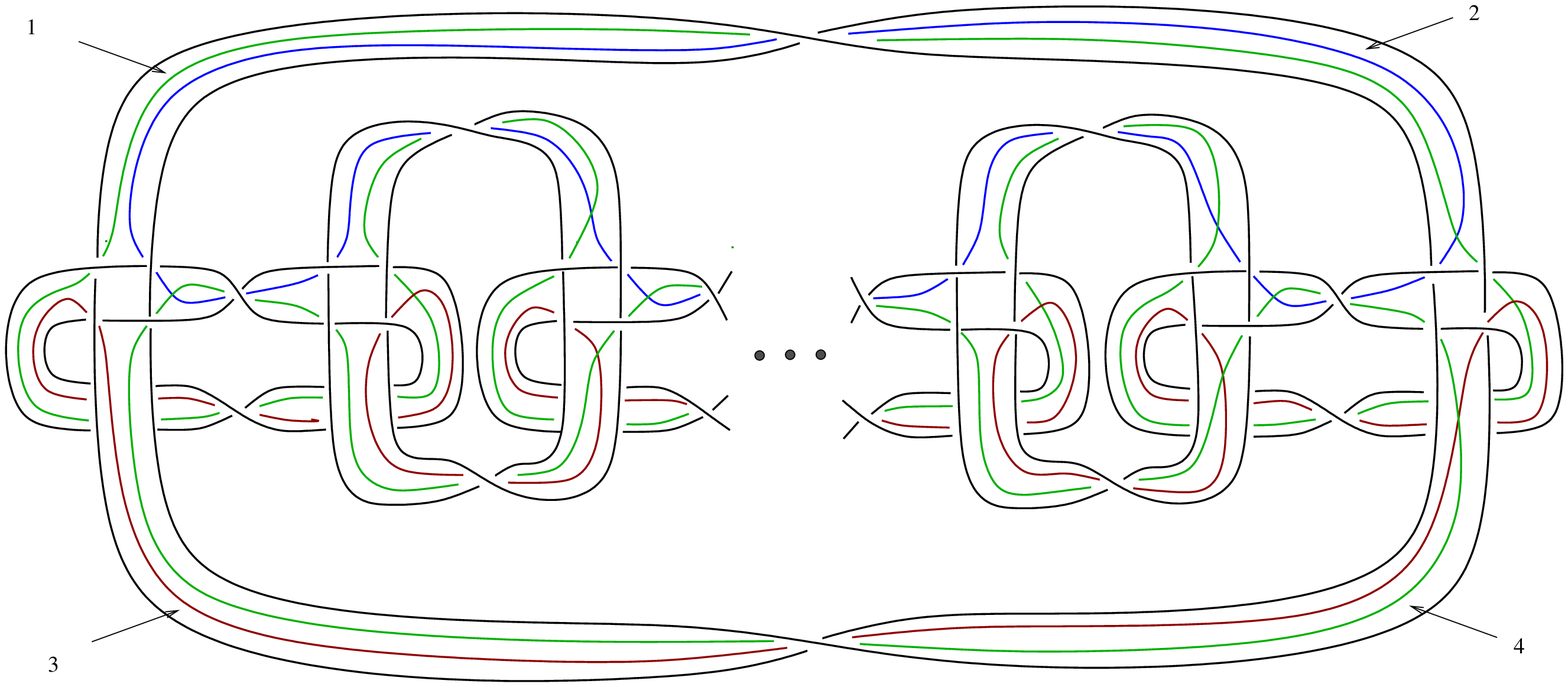}}}
\relabel{1}{{$\a_2$}}
  \relabel{2}{{$\g_1$}}
  \relabel{3}{{$\g_2$}}
  \relabel{4}{{$\a_1$}}

   \endrelabelbox
                        \caption{The vanishing cycles $\a_1, \a_2, \g_1, \g_2$.}
                        \label{fig: roundvanishgenusg}
                    \end{center}
    \end{figure}

Moreover, the monodromy of this Lefschetz fibration $DT^* \S \to D^2$ is given by the product of $2g+6$ positive Dehn twists along the curves $\a_1,  \a_2,  \b_1,  \b_2, \ldots, \b_{2g+2},  \g_1,  \g_2$. The curves $\b_1, \b_2, \ldots, \b_{2g+1}, \b_{2g+2}$ are the core circles of the roundabouts in Figure~\ref{fig: roundvanishdeneme}. The curves $\a_1$ and $\a_2$ are given as the boundary of the two white regions in the checkerboard  coloring we fixed above, while the curves $\g_1$ and $\g_2$ are  given as the boundary of the two black regions. We depicted the curves $\a_1, \a_2, \g_1, \g_2$ in Figure~\ref{fig: roundvanishgenusg}.

The proofs of Theorem~\ref{thm: gene} and Theorem~\ref{thm: isomgene}  are completely analogous to the proofs of Proposition~\ref{prop: tor} and  Proposition~\ref{prop: isom}, respectively. One crucial observation is that $\g_1$ and $\g_2$  are obtained as a result of the simultaneous surgery of $\a_1 \cup \a_2$ and $\b_1 \cup \b_2 \cup \cdots \cup \b_{2g+2}$, at each point where they meet.
The reader can verify this directly for the curves depicted in Figures~\ref{fig: roundvanishdeneme} and ~\ref{fig: roundvanishgenusg}. This fact can also be verified as follows: The $\a$ curves
are given as the boundaries of the two zero handles on $\S$. The $\b$ curves are
then formed from the union of the cores of the one handles joined together in
the zero handles. The result of surgery between the $\a$  and $\b$ curves yields a
curve isotopic to the boundary of the union of the $0$-handles and $1$-handles, in
other words the attaching circles for the $2$-handles. This also explains how the
constructions of Ishikawa and Johns relate to the Morse function on $\S$  in the
same way (see \cite[page 69]{jo}).

\begin{theorem} \label{thm: gene} For any integer $g \geq 0$,  the disk cotangent bundle $DT^* \S$ admits a Lefschetz fibration over $D^2$,  whose regular fiber is a genus one surface with $4g+4$ boundary components. The monodromy of this Lefschetz fibration is given by the composition of positive Dehn twists $$ D(\a_1) D(\a_2) D(\b_1) D(\b_2)\cdots D(\b_{2g+2}) D(\g_1) D(\g_2)$$  where the curves $\b_1, \b_2, \ldots, \b_{2g+1}, \b_{2g+2}$ are shown in Figure~\ref{fig: roundvanishdeneme} and $ \a_1,  \a_2,   \g_1,  \g_2$ are shown in Figure~\ref{fig: roundvanishgenusg}.
\end{theorem}

\begin{theorem} \label{thm: isomgene} For any integer $g \geq 0$, the Lefschetz fibration $DT^*\S \to D^2$ described in Theorem~\ref{thm: gencase} and  the Lefschetz fibration $DT^*\S \to D^2$ described in Theorem~\ref{thm: gene} are isomorphic.
\end{theorem}

The next result immediately follows from Theorem~\ref{thm: isomgene} and Corollary~\ref{cor: gen}.

\begin{corollary} \label{cor: isom} For any integer $g \geq 0$,  the open book decomposition on  $ST^*\S$ induced by the Lefschetz fibration  $DT^*\S \to D^2$ described in Theorem~\ref{thm: gencase} is isomorphic to the open book decomposition on  $ST^*\S$ induced by the Lefschetz fibration $DT^*\S \to D^2$ described in Theorem~\ref{thm: gene}. Therefore, the open book decomposition on $ST^*\S$  induced by the Lefschetz fibration $DT^*\S \to D^2$ described in Theorem~\ref{thm: gene} supports  $\xi_{can}$ as well. \end{corollary}

Finally, we would like to list some questions that arise from the discussion in this paper.

1)  Ishikawa does not give any information about the contact structures on $ST^*\S$ adapted to the open book decompositions which are filled  by the various Lefschetz fibrations he constructs on $DT^*\S$, depending  on different possible choices of an admissible divide on $\S$. Our Corollary~\ref{cor: isom} shows that for a certain admissible divide on $\S$, Ishikawa's open book decomposition on $ST^*\S$ supports the canonical contact structure $\xi_{can}$. Is it true that any open book decomposition on $ST^*\S$ given by Ishikawa's construction supports $\xi_{can}?$

2) Is the Lefschetz fibration on $DT^*\S$ of Johns (which uses the standard
Morse function on the surface) an explicit stabilization of the Lefschetz fibration in this
paper? Is there a calculus relating stabilizations and handle slides of Morse functions
on $\S$ with stabilizations and Hurwitz moves on Lefschetz fibrations on $DT^*\S$?

3) Is it true that the Lefschetz fibration that Johns constructs on $DT^*\S$ is isomorphic to that Ishikawa constructs
for all Morse functions on $\S$, not just the specific Morse function that gives the
minimal genus Lefschetz fibration?

4)  Is the number of boundary components in this article minimal amongst all
genus one Lefschetz fibrations built from the constructions of Johns/Ishikawa varying
over different Morse functions on the surface?

%{\bf {Acknowledgement}}: We would like to thank R. \.{I} Baykur and P. Massot for helpful comments on a draft of this paper.

%clearpage

\end{document}